\documentclass[11pt]{article}
\usepackage{hyperref}	
\usepackage{geometry,amsmath,graphicx,color,ulem,mathrsfs,color}        
\usepackage[version=3]{mhchem}
\usepackage{amssymb,amsfonts,amsmath,comment}
\usepackage{amsthm,enumerate}
\usepackage{color,subfig,tikz,pgf}
\usetikzlibrary{snakes}
\usetikzlibrary{arrows} 

\usepackage{amsthm}
	\newtheorem{theorem}{Theorem}[section]

	\newtheorem{corollary}[theorem]{Corollary}
	\theoremstyle{definition}
		\newtheorem{definition}[theorem]{Definition}
		\newtheorem{example}[theorem]{Example}
		\newtheorem{remark}[theorem]{Remark}
\usepackage{enumitem}

\usepackage{mathtools}

\newenvironment{enumerate*}[1][{}]{\begin{itemize}\renewcommand{\labelitemi}{#1}}{\end{itemize}}

\newcommand\mc[1]{\mathcal{#1}}

\newcommand{\vv}[1]{{\boldsymbol{#1}}}
\newcommand{\df}[1]{{\bf{\emph{#1}}}}
\newcommand{\rrpp}{\mathbb{R}_{>0}}
\newcommand{\rrp}{\mathbb{R}_{\geq 0}}
\newcommand{\rr}{\mathbb{R}}

\begin{document}

\title{Delay stability of reaction systems}
\author{
        Gheorghe Craciun\thanks{Department of Mathematics and Department of Biomolecular Chemistry, University of Wisconsin-Madison.} 
    \and
        Maya Mincheva\thanks{Department of Mathematical Sciences, Northern Illinois University.} 
    \and
        Casian Pantea\thanks{Department of Mathematics, West Virginia University.} 
    \and
        Polly Y. Yu\thanks{Department of Mathematics, University of Wisconsin-Madison.} 
}

\maketitle

\begin{abstract}
    Delay differential equations are used  as a model when the effect of past states has to be taken into account. In this work we consider delay models of chemical reaction networks with mass action kinetics. We obtain a sufficient condition for absolute delay stability of equilibrium concentrations, i.e., local asymptotic stability independent of the delay parameters. Several interesting examples on sequestration networks with delays are presented.
\end{abstract}

\section{Introduction}
\label{sec:intro}

Many biochemical processes involve time delays, for example transmission of cellular signal~\cite{Macdonald1989}. Models using ordinary differential equations assume that future behavior of the system depends only on the present time. Taking into account the influence of the past requires the use of time delay. Delay differential equations have found applications in biology~\cite{Macdonald1989,Smith2011}, population dynamics~\cite{Kuang1993}, chemistry~\cite{Epstein1990,Epstein1991,mr96} and physics~\cite{Stepan1989}. In biochemistry, delay models are sometimes used when the full reaction network is not completely known~\cite{Epstein1990, Epstein1991}, or experimental data displays oscillatory behavior~\cite{Smolen2001}. It is employed most often to model gene regulatory networks, where the delays account for transcription and translation times~\cite{Mahaffy1984}.

If all of the time delays are zero, we obtain the system's ordinary differential equations counterpart. Both the delay and the ordinary differential equations systems have the same set of equilibria. The introduction of delays often, but not always, leads to a stable equilibrium being destabilized~\cite{Cooke1982}, where the delay instability is usually accompanied by a Hopf bifurcation and the appearance of oscillations~\cite{Belair1996}. In that respect, we obtain a sufficient condition which precludes the appearance of oscillations in delay models of chemical reaction networks.

Some delay systems of chemical reaction networks, e.g.~complex balanced mass action systems, are always locally asymptotically stable~\cite{liptak2018semistability}. Our work does not impose any stringent condition like complex balanced; instead we consider fully open mass action systems, i.e., there are inflow and outflow reactions for every chemical species~\cite{Craciun&Feinberg}, and time delays appearing only in the production of chemical species~\cite{mr96}. Based on previous work~\cite{HofbauerSo2000} on delay systems, we obtained a delay-independent algebraic condition for linear stability of mass action systems with delays.

This work is organized as follows. In Section~\ref{sec:notation}, we introduce mass action reaction systems with and without delay. In Section~\ref{sec:linear}, we follow the standard approach of linearizing delay systems and provide a sufficient condition for the linear asymptotic stability of a delay mass action system. In Theorem~\ref{prop:main}, we prove that if a modified version of the Jacobian matrix is a $P_0$-matrix, then the delay system is linearly stable. Furthermore, the conditions of the theorem are independent of the delay parameters. Finally, we conclude with several biologically relevant examples and a discussion in Sections~\ref{sec:examples} and \ref{sec:concl} respectively.

\section{Mass action systems with delay}
\label{sec:notation}

Given a vector $\vv y \in \mathbb{R}^n$, we denote by $\mathrm{supp}(\vv y)$ the set of indices for which $y_i \neq 0$. The cardinality of a set $X$ is denoted $|X|$. Consider the partial order on $\mathbb{R}^n$: if $\vv u$, $\vv v \in \mathbb{R}^n$, then $\vv u \leq \vv v$ if and only if $u_i \leq v_i$ for all $1 \leq i \leq n$. Strict inequality between vectors is similarly defined. Let $\rrp^n$ denote the set of vectors $\vv v \in \mathbb{R}^n$ such that $\vv v \geq \vv 0$. Similarly, $\rrpp^n$ denotes the set of vectors $\vv v \in \mathbb{R}^n$ such that $\vv v > \vv 0$. Given two vectors $\vv  x \in \rrpp^n$ and $\vv y \in \mathbb{R}^n$, we denote by $\vv x^{\vv y}$ the product
	\begin{align*}
		\vv x^{\vv y} = x_1^{y_1} x_2^{y_2} \cdots x_n^{y_n}.
	\end{align*}
\medskip

\begin{definition}
\label{def:CRN}
     A \df{chemical reaction network} $\mathcal{N} = (\mc V, \mc R)$, or \df{reaction network}, is a finite directed graph, where each vertex $\vv y \in \mc V$, called a \df{complex}, is a vector in $\rr_{\geq 0}^n$. Each edge $(\vv y, \vv y') \in \mc R$, called a \df{reaction}, is denoted $\vv y \to \vv y'$. 
\end{definition}

\begin{remark}
The definition above is equivalent to the classical definition of a reaction network being a triple $(\mc S, \mc C, \mc R)$, where $\mc S$ is the set of \emph{species}, $\mc C$ is the set of \emph{complexes} and $\mc R$ is the set of \emph{reactions}~\cite{Banaji_Craciun_2010, Craciun&Feinberg, cf06, feinberg_lecture_notes,  Yu_Craciun_survey}. Indeed, given $\mathcal{N}$ as above, the set of \df{species} is identified (by an abuse of notation) to the standard basis $\{\vv e_1, \vv e_2,\ldots, \vv e_n\}$ of $\rr^n$, and the complexes are non-negative linear combinations of the species. Conversely, given a triple $(\mc S, \mc C, \mc R)$ as described in \cite{Craciun&Feinberg, feinberg_lecture_notes}, we can use the same identification between the set of species and the standard basis of $\rr^n$ to write the complexes as vectors $\vv y \in \rr_{\geq 0}^n$. If $i \in \mathrm{supp}({\vv y})$ then we say that $\vv e_i$ is a species in the complex $\vv y$. 
\end{remark}

For any reaction $\vv y \to\vv  y'$, we call the source vertex $\vv y$ a \df{reactant complex}, and the target vertex $\vv y'$ a \df{product complex}. A species in $\vv y$ is a \df{reactant species} of the reaction $\vv y \to\vv  y'$, and a species in $\vv y'$ is a \df{product species} of the reaction. In other words, $\mathrm{supp}(\vv y)$ consists of the reactant species while $\mathrm{supp}(\vv y')$ consists of the product species.

A reaction of the form $A \to 0$ is  an \df{outflow reaction}, and a reaction of the type $0 \to A$ is an \df{inflow reaction}. A reaction $\vv y \to \vv y'$ is \df{autocatalytic} if $\mathrm{supp}(\vv y) \cap \mathrm{supp}(\vv y') \neq \emptyset$ and for every $i \in \mathrm{supp}(\vv y) \cap \mathrm{supp}(\vv y')$, we have $y_i' > y_i $. Thus for a \df{non-autocatalytic network}, i.e., one with no autocatalytic reactions, we have $y_i' \leq y_i$ for all reaction $\vv y\to \vv y'$ and any $i \in \mathrm{supp}(\vv y) \cap \mathrm{supp}(\vv y')$.

A model comes with a reaction network, by assuming that each reaction proceeds at a certain rate. A \emph{kinetics} for a reaction network $\mathcal{N} = (\mc V, \mc R)$ is an assignment of a rate function $\mathscr{K}_{\vv y\to\vv  y'}: \rrpp^{n} \to \rrpp$ to each reaction $\vv y \to \vv y' \in \mc R$. One of the most common models in the literature for chemistry and biochemistry is that of \emph{mass action kinetics}~\cite{Guldberg1864, HornJackson1972}. Although we briefly mention results for a general class of kinetics, in this paper the focus is mass action. Mass action kinetics assumes that the rate at which a reaction $\vv y \to \vv y'$ proceeds is proportional to the concentrations of the reactant species, i.e., at rate $\mathscr{K}_{\vv y \to \vv y'}(\vv x) = k_{\vv y \to  \vv y'} \vv x^{\vv y}$, where $k_{\vv y \to \vv y'} > 0$ is a positive rate constant and  $\vv  x = \vv  x(t) \in \rrpp^{n}$ is the (time-dependent) vector of concentrations.

\begin{definition}
	A \df{mass action system} $\mc N_{\vv k}$ is a reaction network $\mc N = (\mc V, \mc R)$ together with a vector of positive rate constants $\vv k \in \rrpp^{\mc R}$.  The dynamics of the concentration vector $\vv x(t)$ is given by
	\begin{align}
	\label{eq:gode-mak}
		\dot{\vv x}(t) &= \sum_{\vv y\to \vv y' \in \mc R} k_{\vv y\to \vv y'} [\vv x(t)]^{\vv y} (\vv y'-\vv y).
	\end{align}
\end{definition}

There are chemical systems where the products are not immediately produced though the reactants are consumed. For example, biopolymer processes that involve a nucleation-propagation mechanism, like the binding of two single-strand DNA molecules $2S \to D$. In a model of this process, the consumption of $S$ happens with no delay, but $D$ becomes available after a time delay  $ \tau > 0$. In an upcoming work, we show that this model of duplex formation is delay stable~\cite{our_follow_up_paper_SR_graph}.

In a delay model for realistic chemical and biochemical systems, \emph{the delay terms only affect product formation}. This is consistent with the physical intuition that reactant species are generally consumed immediately, while the product species are available only at a later time. While different products of the same reaction may be available after two different delays, for simplicity of notation, we assign at most one delay parameter to each reaction. As we note later, our main result Theorem~\ref{prop:main} does not depend on this assumption. 

Let $\vv \tau = (\tau_{\vv y \to \vv y'})_{\vv y \to \vv y' \in \mathcal{R}} \in \rrp^{\mathcal R}$ be a vector of time delays for the reactions in $\mathcal{N}$. If $\tau_{\vv y\to \vv y'} = 0$, we say the reaction $\vv y \to \vv y'$ occurs without delay. Note that inflow and outflow reactions always occur without delay. In the case when a reaction is delayed, the rate function for product formation is evaluated at a shifted time $ k_{\vv y\to \vv y'} [\vv x(t-\tau_{\vv y\to \vv y'})]^{\vv y}$. These considerations lead us to the definition of a delay mass-action system, as introduced by Roussel in \cite{mr96}.

\begin{definition}
A \df{delay mass action system} $\mc N_{\vv \tau,\vv  k}$ is a mass action system $\mathcal{N}_{\vv k}$ with a vector of delays $\vv\tau \in \rrp^{\mathcal R}$. The dynamics of the concentration vector $\vv x(t)$ is given by 
\begin{align}
	\label{eq:gode-delay}
		\dot{\vv x}(t) &= \sum_{\vv  y\to \vv y' \in \mc R} k_{\vv y\to \vv y'} [\vv x(t- \tau_{\vv y\to \vv y'})]^{\vv y} \vv y' - \sum_{\vv y\to \vv y' \in \mc R}
		 k_{\vv y\to \vv y'} [\vv x(t)]^{\vv y} \vv y.
	\end{align}
\end{definition}

For an initial value problem, the initial data has to be specified on the interval $[-\bar{\tau}, 0]$ where $\bar{\tau} = \max_{\vv y \to \vv y'} \tau_{\vv y\to\vv  y'}$. If all reactions occur without delay, i.e., $\vv \tau=\vv 0$, then from the perspective of dynamics, $\mathcal{N}_{\vv \tau, \vv k}$ is not different from $\mathcal{N}_{\vv k}$~\cite{Banks_2013}. Indeed, the ODE system (\ref{eq:gode-mak}) is identical to the delay system  (\ref{eq:gode-delay}) when $\vv\tau = \vv 0$.

It is well-known that the ODE system (\ref{eq:gode-mak}) has only non-negative solutions if the initial condition is non-negative~\cite{liptak2018semistability}. The first quadrant for a delay system such as (\ref{eq:gode-delay}) is also forward invariant~\cite{Bodnar2000}. The systems (\ref{eq:gode-mak}) and (\ref{eq:gode-delay}) share the same set of positive steady states~\cite{liptak2018semistability}. In other words, a positive constant solution $\vv x(t) \equiv \vv x^*$ is a steady state for the delay system (\ref{eq:gode-delay}) if and only if it is a steady state of the ODE system (\ref{eq:gode-mak}). We call a positive steady state $\vv x^*$ an \df{equilibrium}. 

\begin{remark}
In general, the associated ODE equations \eqref{eq:gode-mak} of a mass-action system, with positive initial data $\vv \theta \in \rrpp^n$, may have a conservation relation
    \begin{align*}
        \vv x(t) - \vv \theta \in S, 
    \end{align*}
where $S = \mathrm{span} \{ \vv y' - \vv y : \vv y \to \vv y' \in \mc R\}$ is the \df{stoichiometric subspace}.\footnote{Generally we say the system has such a conservation relation if $S \subsetneq \rr^n$.}  
Similarly, the delay equations \eqref{eq:gode-delay} of a delay mass-action system, with continuous and positive initial data $\vv \theta$ defined on the interval $[-\bar{\tau},0]$, may admit a conservation relation~\cite{liptak2018semistability}
    \begin{align*} 
        \vv x(t) - \boldsymbol{\theta}(0) + \sum_{\vv y \to \vv y' \in \mc R} k_{\vv y \to \vv y'}  \left(  \int_{t-\tau_{\vv y \to \vv y'}}^t  [\vv x(s)]^{\vv y}  \, ds - \int_{-\tau_{\vv y \to \vv y'}}^0 [\vv \theta(s)]^{\vv y}   \, ds \right) \vv y  \in S.
    \end{align*}

While the ODE and delay models share the same set of positive equilibria, when $S \subsetneq \rr^n$, solving for an equilibrium with a given initial data can be difficult. In this paper, we only consider systems whose stoichiometric subspace $S$ is the whole $\rr^n$. 
\end{remark}

We represent a mass action system $\mathcal{N}_{\vv k}$ or a delay mass action system $\mathcal{N}_{\vv \tau, \vv k}$ by labeling the reactions with their rate constants and (if non-zero) delay parameters. A reaction $\vv y \to \vv y'$ with rate constant $k > 0$ that occurs without delay, i.e., $\tau_{\vv y\to \vv y'} = 0$, is shown as 
    \begin{align*}
        \vv y \xrightarrow{k} \vv y',
    \end{align*}
while a reaction with rate constant $k > 0$ and delay parameter $\tau_{\vv y \to \vv y'}  = \tau > 0$ will be shown as 
    \begin{align*}
        \vv y \xrightarrow[\tau]{k} \vv y'.
    \end{align*}
Sometimes we omit the labels altogether; in this case, it will be clear from the context whether we are referring to the reaction network $\mathcal N$, or the mass action system $\mathcal N_{\vv k}$ or the delay mass action system $\mathcal N_{\vv \tau, \vv k}$.

\begin{example}
\label{ex:runningEx}
Consider a reaction network $\mc N$ with three species $A_1$, $A_2$, $A_3$. The network consists of the following reactions:
    \begin{align*}
	\label{eq:runningEx}
		A_1 + A_2 &\xrightarrow{k_1}  A_3, \quad
		A_1 \xrightarrow{k_2}  A_2, \quad 
		0 \xrightarrow{k_3}  A_1, 
		\\
		A_1 &\xrightarrow{k_4} 0, \quad 
		A_2  \xrightarrow{k_5}   0 ,  \quad 
		A_3 \xrightarrow{k_6} 0.  \nonumber
	\end{align*} 
The system of ODEs for this mass action system
    \begin{align*}
        \dot{x}_1 &= -k_1 x_1 x_2 - k_2 x_1 + k_3 - k_4 x_1 \\
        \dot{x}_2 &= -k_1 x_1 x_2 + k_2 x_1 - k_5 x_2 \\
        \dot{x}_3 &= k_1 x_1 x_2 - k_6 x_3 
    \end{align*}
admits a unique positive equilibrium. 

Consider delay $\tau_{1} \geq 0$ for the reaction $A_1+A_2 \to A_3$ and  $\tau_{2} \geq 0$ for the reaction $A_1 \to A_2$. To simplify notation in the system of delay equations, when the concentration function $x_i$ is not shifted in time, i.e., $x_i = x_i(t)$, we suppress the explicit appearance of time. The system delay differential equations for $\mc N_{\vv\tau, \vv k}$ is
    \begin{align*}
        \dot{x}_1 &= -k_1 x_1 x_2 - k_2 x_1 + k_3 - k_4 x_1 \\
        \dot{x}_2 &= -k_1 x_1 x_2 + k_2 x_1(t-\tau_2) - k_5 x_2 \\
        \dot{x}_3 &= k_1 x_1(t-\tau_1) x_2(t-\tau_1) - k_6 x_3 .
    \end{align*}
\end{example}

\section{Linear analysis of delay reaction systems}
\label{sec:linear}

In this section, we introduce the linearization of the delay mass action system (\ref{eq:gode-delay}) about an equilibrium $\vv x^*$. For a fixed $\vv x > \vv 0$, we define a special dot product:   
    \begin{align*}
        \vv y*\vv e_i = \frac{\vv y \cdot \vv e_i}{\vv x \cdot \vv e_i} = \frac{y_i}{x_i}, 
    \end{align*}
where $\{\vv e_i : 1 \leq i \leq n \}$ is the standard basis of $\rr^n$. The product depends on a choice of $\vv x > \vv 0$, which will be evident from context, for example, as in the partial derivative $\partial_i \vv x^{\vv y} = \vv x^{\vv y}(\vv y * \vv e_i)$~\cite{Craciun&Feinberg}.

The linearization of the delay system (\ref{eq:gode-delay}) is done the usual way, by adding a small quantity $\delta \vv x $ to an equilibrium $\vv x^*$. Linearizing the delay system (\ref{eq:gode-delay})  about $\vv x^*$, we  obtain the linear delay system
\begin{equation}\label{eq:del_lin_sys}
    \dot{\vv x} =  \sum_{\vv y \to \vv y' \in \mathcal{R}} J_{\tau_{\vv y \to \vv y'}} (\vv x^*,\vv k) \vv x (t - \tau_{\vv y \to \vv y'}) -\sum_{\vv y\to \vv  y'  \in \mathcal{R}} J_{\vv y\to \vv y'} (\vv x^*,\vv k)\vv x(t),
\end{equation}
where
\begin{equation*}\label{jac-tauy}
    J_{\tau_{\vv y \to \vv y'}} (\vv x,\vv k) = \left[ k_{\vv y \rightarrow \vv y'}\vv x^{\vv y} (\vv y*\vv e_1)\vv y',\, \ldots , \,k_{\vv y \rightarrow \vv y'}\vv x^{\vv y} (\vv y*\vv e_n)\vv y'   \right]  
\end{equation*} 
and 
\begin{equation*}\label{eq:jac0}
    J_{\vv y\to \vv y'}(\vv x,\vv k) = \left[ k_{\vv y \rightarrow \vv y'} \vv x^{\vv y} (\vv y*\vv e_1)\vv y,\, \ldots ,\, k_{\vv y \rightarrow \vv y'}\vv x^{\vv y} (\vv y*\vv e_n)\vv y   \right]
\end{equation*} 
are $n \times n$ matrices.

Next substitute $\vv x(t) - \vv x^* = \vv a e^{\lambda t}$ into the linearized system (\ref{eq:del_lin_sys}) to obtain its characteristic equation. The resulting linear system (in $\vv a$) is
\begin{equation*}\label{eq:ld1}
    \left(\sum_{\vv y \to \vv y' \in \mathcal{R} }
    \left(
     J_{\tau_{\vv y \to\vv  y'}} (\vv x^*,\vv k) e^{-\lambda \tau_{\vv y \to \vv y'}} - 
    J_{\vv y\to \vv y'} (\vv x^*,\vv k) 
    	\right)
    -\lambda I \right)\vv  a = \vv 0,
\end{equation*}
which has a non-zero solution if and only if 
\begin{equation}\label{eq:char_pol_del}
    \det ( J_{\lambda} (\vv x^*,\vv k, \vv \tau) - \lambda I ) = 0,
\end{equation}   
where 
\begin{align}\label{eq:jlam}
    &J_{\lambda} (\vv x,\vv k, \vv \tau) =
    \sum_{\vv y \to \vv y'  \in \mathcal{R}} J_{\tau_{\vv y \to\vv  y'}} (\vv x,\vv k) e^{-\lambda \tau_{\vv y \to \vv y'}} - \sum_{\vv y \to\vv  y'  \in \mathcal{R}}J_{\vv y\to\vv  y'} (\vv x,\vv k)    \nonumber   \\
    =&\left[ \sum_{\vv y \rightarrow \vv y' \in {\mathcal R}}{k_{\vv y \rightarrow \vv y'}\vv x^{\vv y} (\vv y*\vv e_1)(\vv y'e^{-\lambda \tau_{\vv y\to\vv  y'}} - \vv y)},\, \ldots , \sum_{\vv y \rightarrow \vv y' \in {\mathcal R}}{k_{\vv y \rightarrow \vv y'} \vv x^{\vv y} (\vv y*\vv e_n)(\vv y' e^{-\lambda \tau_{\vv y\to\vv  y'}} -\vv y)}   \right]  .
\end{align}
The transcendental equation (\ref{eq:char_pol_del}) is the \df{characteristic equation} of  (\ref{eq:del_lin_sys}).

If $\vv \tau = \vv 0$, we recover the Jacobian matrix of the corresponding ODE system (\ref{eq:gode-mak}), i.e., we have  $J_{\lambda}(\vv x^*, \vv k, \vv\tau= \vv 0) = J(\vv x^*, \vv k)$, where
\begin{equation}\label{eq:jac}
    J (\vv x,\vv k) = \left[ \sum_{\vv y \rightarrow \vv y' \in {\mathcal R}} k_{\vv y \rightarrow \vv y' } \vv x^{\vv y} (\vv y*\vv e_1) (\vv y'-\vv y) ,\,  \ldots , 
    \sum_{\vv y \rightarrow \vv y' \in {\mathcal R}} k_{\vv y \rightarrow \vv y'}  \vv x^{\vv y} (\vv y*\vv e_n)(\vv y' -\vv y)  \right] .
\end{equation}

An equilibrium  $\vv x^* \in \rrpp^{n}$ of a delay mass action system $\mathcal{N}_{\vv \tau, \vv k}$ is \emph{asymptotically  stable} if the characteristic equation (\ref{eq:char_pol_del}) has only roots  $\lambda$ with negative real parts. The equilibrium  $\vv x^*$ is \emph{unstable} if at least one root of (\ref{eq:char_pol_del}) has positive real part~\cite{Hale1993, Smith2011}.

Theorem~\ref{prop:main} gives a sufficient condition on the modified Jacobian matrix (defined below) for the asymptotic stability of an equilibrium of the delay system (\ref{eq:gode-delay}) and of the linearized system (\ref{eq:del_lin_sys}), \emph{independent of the choice of delays}. Asymptotic stability of an equilibrium {independent} of the delay parameters is known as \emph{absolute stability}, as opposed to \emph{conditional stability}, where asymptotic stability depends on the delay parameters~\cite{Bellman_Cooke,Brauer}. In light of this, we define the notion of absolute stability for a mass action system, and the stronger notion of delay stability for a reaction network (under mass action kinetics).

\begin{definition}
     \label{df:ASmas}
     A delay mass action system $\mathcal{N}_{\vv \tau, \vv k}$ is said to be \df{absolutely stable} if for any positive equilibrium, every root $\lambda$ of the characteristic equation (\ref{eq:char_pol_del}) has negative real part, for any choice of delay parameters $\vv \tau \geq \vv 0$.
\end{definition}

\begin{definition}
\label{df:DScrn}
    A reaction network $\mathcal{N}$ is \df{delay stable} if the delay mass action system $\mathcal{N}_{\vv\tau , \vv k}$ is absolutely stable for any choices of $\vv k > \vv 0$ and $\vv \tau \geq \vv 0$.
\end{definition}

\section{Main Result}
\label{sec:mainresult} 

In our main result, the following matrix $\tilde{J}(\vv x,\vv k)$ plays an important role: 
\begin{equation}\label{jactilde}
    \left[ \sum_{\vv y \rightarrow \vv y' \in {\mathcal R}}{k_{\vv y \rightarrow \vv y'}\vv x^{\vv y} (\vv y*\vv e_1)(\vv y'+\tilde{\vv y}^{(1)})}, \ldots , \sum_{\vv y \rightarrow \vv y' \in {\mathcal R}}{k_{\vv y \rightarrow \vv y'}{\vv x}^{\vv y} (\vv y*\vv e_n)(\vv y'+\tilde{\vv y}^{(n)})}   \right]  
\end{equation}
where  $\tilde{\vv y}^{(i)}=(y_1, y_2,\ldots, -y_i, \ldots ,y_n)^\top.$ The matrix (\ref{jactilde}) looks similar to the Jacobian matrix $J(\vv x, \vv k)$, but in the $i$th column, instead of the reaction vectors $\vv y' - \vv y$, we have $\vv y' + \vv y^{(i)}$. A change of sign occurs in every off-diagonal component of the reactant complex. We call the matrix $\tilde{J}(\vv x,\vv k)$ the \df{modified Jacobian matrix} of the network.  

\medskip

The matrices $J_{\lambda}(\vv x^*,\vv k, \vv \tau)$, $J(\vv x^*,\vv k)$,  $\tilde{J}(\vv x^*,\vv k)$ and the characteristic equation (\ref{eq:char_pol_del}) give information about stability of any equilibrium $\vv x^*$ of a delay mass action system $\mathcal N_{\vv \tau, \vv k}$. However, the matrices and the characteristic equation are well-defined at any positive state $\vv x$ and any positive rate constants $\vv k$.

\begin{example}
\label{ex:runningEx2}
Consider the delay mass action system from Example~\ref{ex:runningEx}. Its Jacobian matrix is
    \begin{align*}
        J(\vv x, \vv k) = \left[ 
        \begin{matrix}
          -(k_1 x_2+k_2 + k_4) & -k_1 x_1 & 0 \\
          -k_1x_2 + k_2 & -(k_1 x_1+k_5) & 0 \\
          k_1x_2 & k_1 x_1 & - k_6
        \end{matrix}
        \right].
    \end{align*}
For the delay system, we have
    \begin{align*}
        J_\lambda(\vv x, \vv k, \vv \tau) = \left[ 
        \begin{matrix}
            -(k_1 x_2+k_2 + k_4) & -k_1 x_1 & 0 \\
          -k_1x_2 + k_2e^{-\lambda \tau_2} & -(k_1 x_1+k_5) & 0 \\
          k_1x_2 e^{-\lambda \tau_1} & k_1 x_1 e^{-\lambda \tau_1} & - k_6
        \end{matrix}
        \right],
    \end{align*}
and the characteristic equation of the linearized system is 
\begin{align*}
    0 &= \det(J_\lambda - \lambda I) 
    = \lambda^3 + \lambda^2 \left( k_1(x_1+x_2) + k_2 + k_4 + k_5 - k_6 \right) 
    \\& \quad + \lambda \left( k_1k_2x_1 + k_1k_4x_1 
    -k_1k_6x_1 + (k_2-k_6)(k_1x_2+k_2+k_4) - k_5k_6\right)
    \\& \quad + k_1^2k_6 x_1x_2
    + e^{-\lambda \tau_2} \left( \lambda k_1 x_2 - k_1k_6x_1\right). 
\end{align*}  
Finally, the modified Jacobian matrix is 
    \begin{align*}
        \tilde{J}(\vv x, \vv k) = \left[ 
        \begin{matrix}
          -(k_1 x_2+k_2 + k_4) & k_1 x_1 & 0 \\
          k_1x_2 + k_2 & -(k_1 x_1+k_5) & 0 \\
          k_1x_2 & k_1 x_1 & - k_6
        \end{matrix}
        \right].
    \end{align*}
\end{example}

A matrix $M$ is a \emph{$P_0$-matrix} if it has only non-negative principal minors~\cite{fiedlerptak1962,johnson1974}. A matrix $M$ is \emph{reducible} if it can be placed into block upper triangular form by simultaneous row and column permutations; otherwise $M$ is \emph{irreducible}~\cite{Lancaster1985}. The following theorem is inspired by~\cite[Lemma~1]{HofbauerSo2000}. 
Note that one of the hypotheses in this theorem is the absence of autocatalytic reactions; e.g. reactions of the form $A \to 2A$ are forbidden. In a non-autocatalytic network, we have $y'_i \leq y_i$ for any reaction $\vv y \to \vv y'$ and any $i \in \mathrm{supp}(\vv y) \cap \mathrm{supp}(\vv y')$. Another restriction comes from $\det(J) \neq 0$. In particular, the stoichiometric subspace is $S = \rr^n$, so the system admits no conservation relation.

\begin{theorem}
\label{prop:main}
Let $\mathcal N$ be a non-autocatalytic network, and $\mathcal N_{\vv \tau, \vv k}$ be the delay mass action system for some $\vv k > \vv 0$ and $\vv \tau \geq \vv 0$. Let $\vv x^* > \vv 0$ be  an equilibrium of $\mathcal N_{\vv \tau, \vv k}$. Let $J_\lambda$, $J$ and $\tilde J$ be defined as in (\ref{eq:jlam}), (\ref{eq:jac}) and (\ref{jactilde}) at $\vv x^*$, $\vv k$ and $\vv\tau$. Suppose $\det J \neq 0$, $\tilde{J}_{ii}<0$ for all $i$, and  $-\tilde{J}$  is a $P_0$-matrix. Then all the roots of the characteristic equation $\det (J_{\lambda} - \lambda I) =0$ have negative real parts. 
\end{theorem}

\begin{proof}
Any root $\lambda$ of the characteristic equation $\det(J_\lambda - \lambda I) = 0$ is non-zero, because $J_{\lambda} = J$ when $\lambda = 0$. Suppose for a contradiction, that for some values of delay, the characteristic equation (\ref{eq:char_pol_del}) has a root $\lambda \neq 0$ with real part $\mathrm{Re}(\lambda) \geq 0$. Such a root $\lambda$ is an eigenvalue of $J_\lambda$. 

For $i\neq k$, we have $|(J_\lambda)_{ik}| \leq \tilde J_{ik}$, and $\tilde{J}_{ik} \neq 0$ if $(J_\lambda)_{ik} \neq 0$. Suppose for now that $\tilde J$ is irreducible. Because $-\tilde J$ is an irreducible $P_0$-matrix, there exists a vector $\vv v>\vv 0$ such that $\tilde{J}  \vv v \leq \vv 0 $~\cite[Theorem 5.8]{fiedlerptak1962}. For $i = 1,2,\ldots, n$, we have
\begin{equation}\label{eq:wdd_c}
    \tilde{J}_{{ii}} v_i +\sum_{k \ne i}  \tilde{J}_{ik}   v_k  \leq 0.
\end{equation}

To apply Gershgorin circle theorem, we are interested in the disks $B_i$ with center $(J_\lambda)_{ii}$ and radius $ v_i^{-1} \sum_{k \neq i} |J_\lambda|_{ik} v_k$. For each $i = 1,2,\ldots, n$, there are two cases to consider. In the first case, suppose $ e^{-\lambda \tau_{\vv y\to \vv y'}} \in \rr$ for every $\vv y \to \vv y'$ with $y_i$, $y_i' \neq 0$. Then $(J_\lambda)_{ii} \leq \tilde{J}_{ii} \leq 0$. From (\ref{eq:wdd_c}), we have 
    \begin{equation*}\label{eq:rad2}
        v_{i}^{-1} \sum_{k \ne i} |(J_{\lambda})_{ik}|  v_k  \leq 
        - \tilde{J}_{ii} \leq -\mathrm{Re}(J_\lambda)_{ii}.
    \end{equation*}

In the second case, suppose there is at least one reaction $\vv y \to \vv y'$ for which $y_i$, $y_i' \neq 0$ and  $e^{-\lambda \tau_{\vv y\to \vv y'}}  \not\in \mathbb{R}$. For every such reaction, $\sin(\mathrm{Im}(\lambda) \tau_{\vv y\to \vv y'}) \neq 0$, or $\cos(\mathrm{Im}(\lambda) \tau_{\vv y\to \vv y'}) < 1$, and thus 
    \begin{align*}
        y_i'  \mathrm{Re}(e^{-\lambda \tau_{\vv y\to \vv y'}}) 
        = y_i'  e^{-\mathrm{Re}(\lambda) \tau_{\vv y \to \vv y'}} \cos(\mathrm{Im}(\lambda)\tau_{\vv y\to\vv y'}) < y_i'. 
    \end{align*} 
Therefore, $\mathrm{Re}(J_\lambda)_{ii} < \tilde J_{ii}$. It follows from (\ref{eq:wdd_c}) that 
\begin{align*}
    \mathrm{Re}(J_\lambda)_{ii} v_i + \sum_{k\neq i} | (J_{\lambda})_{ik}| v_k
    &< \tilde J_{ii} v_i + \sum_{k\neq i} \tilde J_{ik} v_k \leq 0.
\end{align*}
In other words, 
    \begin{equation*}\label{eq:rad}
        v_{i}^{-1}\sum_{k \ne i} |(J_{\lambda})_{ik}|  v_k  < -\mathrm{Re}(J_\lambda)_{ii}.
    \end{equation*}
Therefore, for $i=1,2,\ldots, n$, any non-zero element in the disk $B_i$, with center $(J_\lambda)_{ii}$ and radius $\sum_{k \ne i}v_{i}^{-1} |(J_{\lambda})_{ik}|  v_k$, has negative real part.

Let $D = \mathrm{diag}(v_1,v_2,\ldots, v_n)$, and consider the matrix $D^{-1} J_{\lambda} D$, which is similar to $J_\lambda$ and shares the same eigenvalues. By Gershgorin's theorem, the eigenvalues of $D^{-1} J_{\lambda} D$ are contained in the union of the disks $B_i$. Hence, any non-zero eigenvalue of $J_{\lambda}$ has negative real part. This contradicts our assumption that $\lambda$ has non-negative real part.

Now consider the case when $\tilde{J}$ is reducible. By relabeling the species, we may assume that $\tilde{J}$ is an upper block triangular matrix with irreducible blocks along the diagonal~\cite{BermanPlemmons1994,HofbauerSo2000}.  Of course, the principal minors of $\tilde{J}$ are unchanged. Whenever $i\neq k$, note that $\tilde{J}_{ik} = 0$ implies that $(J_\lambda)_{ik} = 0$, so each irreducible diagonal block of $\tilde J$ corresponds to a (possibly reducible) diagonal  block of $J_\lambda$. In particular, $\det(J_\lambda - \lambda I)$ is the product of $\det(M_j - \lambda I)$, where each $M_j$ is a diagonal block of $J_\lambda$ corresponding to an irreducible diagonal block of $\tilde{J}$. Since $\lambda$ is an eigenvalue of $J_\lambda$, it is an eigenvalue of some $M_j$. Now the result above can be applied to the corresponding irreducible $P_0$-block of $\tilde{J}$.
\end{proof}

\begin{remark}
Although Theorem~\ref{prop:main} is stated for mass action systems, the result holds for more general kinetics under some mild conditions. More precisely, the result above holds for kinetics $\mathscr{K}$ defined on $\rrpp^n$, where $\frac{\partial \mathscr{K}_{\vv y\to \vv y'}}{\partial x_j} \geq 0$ for all indices $j$, and for any $i \in \mathrm{supp}(\vv y)$ we also require that $\frac{\partial \mathscr{K}_{\vv y\to \vv y'}}{\partial x_i} > 0$.
\end{remark}  

\begin{remark}
\label{rmk:difftau}
In the proof above, we did not make use of the fact that the same delay parameter $\tau_{\vv y\to \vv y'}$ could appear more than once in $J_{\lambda}$. The above result holds even when different species are produced by the same reaction with different delay parameters, as in Example~\ref{ex:notdelaystable}. 

For example, we may have a reaction $A_1 \to 2A_2 + A_3$, where $A_2$ is produced after a delay time $\tau_1$ and $A_3$ is produced after a delay time $\tau_2$. The matrix $J_\lambda$ would include in its first column the term 
    \begin{align*}
        k x_1 \begin{bmatrix}
            -1  \\
           2  e^{-\tau_1} \\ 
           e^{-\tau_2}
        \end{bmatrix}.
    \end{align*}
\end{remark} 
\medskip 

When the assumptions in  Theorem~\ref{prop:main} hold independent of the rate constants $\vv k$ and equilibrium $\vv x^*$, we conclude delay stability.

\begin{corollary}
\label{cor:main}
Let $\mathcal N$ be a non-autocatalytic network. Let $J$ and $\tilde J$ be defined as in (\ref{eq:jac}) and (\ref{jactilde}) as functions of $\vv x > \vv 0$ and $\vv k > \vv 0$. Suppose $\det J \neq 0$, $\tilde{J}_{ii}<0$ for all $i$, and  $-\tilde{J}$  is a $P_0$-matrix for all choices of  $\vv k > \vv 0$  and all equilibrium points $\vv x > \vv 0$  of $\mathcal{N}_{\vv k}$. Then $\mathcal{N}$ is delay stable. 
\end{corollary}

Under the hypotheses of  Corollary~\ref{cor:main},
any positive equilibrium $\vv x^*$ of the delay mass action system $\mathcal{N}_{\vv \tau, \vv k}$ is asymptotically stable for any choice of $\vv k > \vv 0$, $\vv \tau \geq \vv 0$.

\section{Examples}
\label{sec:examples}

\begin{example}
\label{ex:ocycleok}
Again, we return to the delay mass action system of Examples~\ref{ex:runningEx} and \ref{ex:runningEx2}, which we claim is delay stable. In other words, for any choice of rate constants $\vv k \in \rrpp^{\mc R}$ and any choice of delay parameters $\vv \tau \in \rrp^{\mc R}$, the delay mass action system $\mc N_{\vv \tau, \vv k}$ is absolutely stable --- any positive equilibrium is linearly stable for the system of delay differential equations. 

Recall that the modified Jacobian matrix is 
\begin{align*}
	\tilde{J}(\vv x, \vv k) = \begin{bmatrix}
		-k_1 x_2 -k_2  - k_4 & k_1 x_1 & 0 \\
		k_1 x_2 + k_2 &  - k_1 x_1 -k_5 & 0 \\
		k_1 x_2 & k_1 x_1 & -k_6
	\end{bmatrix}.
\end{align*}
It is not difficult to check that the principal minors of $-\tilde J(\vv x, \vv k)$ are all positive whenever $\vv x \in \rrpp^n$ and $\vv k \in \rrpp^{\mc R}$. Moreover, all the assumptions  in Corollary~\ref{cor:main} are satisfied. Therefore, the reaction network presented in Example~\ref{ex:runningEx} is delay stable, i.e., for all choices of rate constants $\vv k$ and delay parameters $\vv \tau$, any positive equilibrium is linearly stable for the delay mass action system $\mc N_{\vv \tau, \vv k}$. 
\end{example}

The next two examples contain {\em sequestration reactions}. These are reactions of the type $A_1+A_2 \to P$, where $P$ does not participate in other reactions. One example of sequestration reactions is when $A_1$ is a substrate that ``sequesters" an enzyme $A_2$ by binding it and making it inactive. Inactivation mechanisms containing reactions of this type (with $A_1$ sometimes called a {\em suicide substrate}) have been studied both analytically and experimentally \cite{maini1991}. Moreover, versions of sequestration networks appear as substructures of relevant enzymatic systems \cite{sequestration} and feature remarkable connections between network structure and dynamical behavior \cite{Craciun&Feinberg, sequestration, felix2016}.

\begin{example}
\label{ex:sequest1}
Consider the following sequestration network:
\begin{align}\label{eq:seq6}
    A_1 + A_2 &\to P  \nonumber \\ 
    A_2 +A_3 &\to Q \nonumber \\
    A_3 +A_1 &\to R \nonumber \\
    A_1\rightleftharpoons 0, \quad A_2 &\rightleftharpoons 0 \quad 
    A_3\rightleftharpoons 0 \nonumber\\
    P\rightleftharpoons 0, \quad Q &\rightleftharpoons 0 \quad 
    R\rightleftharpoons 0 \nonumber
\end{align}
With $x_1,x_2, x_3, x_P, x_Q, x_R$ denoting the concentrations of $A_1$, $A_2$, $A_3$, $P$, $Q$, and $R$ respectively, the corresponding delay system is 
\begin{align*}
    \dot{x}_1 &= -k_1 x_1x_2 -k_3 x_3x_1-k_4 x_1+k_5 \\[2pt]
    \dot{x}_2 &= -k_1 x_1x_2 -k_2 x_2x_3-k_6 x_2+k_7 \\[2pt]
    \dot{x}_3 &= -k_2 x_2x_3 -k_3 x_3x_1-k_8 x_3+k_9 \\
    \dot x_P  &=k_1x_1(t-\tau_1)x_2(t-\tau_1)-k_{10}x_P+k_{11} \nonumber\\
    \dot x_Q  &=k_2x_2(t-\tau_2)x_3(t-\tau_2)-k_{12}x_Q+k_{13} \nonumber\\
    \dot x_R  &=k_1x_3(t-\tau_3)x_1(t-\tau_3)-k_{14}x_R+k_{15}. \nonumber
\end{align*}

Note that the modified Jacobian matrix for the augmented system has the block form
\begin{equation*}
    \tilde{J}=
    \begin{bmatrix}
    \tilde {J}_1 &0\\
    A &-D
    \end{bmatrix}
\end{equation*}
where $D=\textrm{diag}(k_{10},k_{11},k_{12})$ is a positive diagonal matrix and 
\begin{equation*}\label{eq:tildeJseq3}
\tilde{J_1}(\vv x, \vv k)=
\begin{bmatrix}
-k_1 x_2 -k_3x_3-k_4 & k_1 x_1 & k_3x_1 \\
k_1 x_2 & -k_1x_1 -k_2x_3 -k_6 & k_2x_2\\
k_3x_3 &k_2 x_3  & -k_2 x_2 -k_3 x_1 -k_8\\
\end{bmatrix}.
\end{equation*}

Let $J_1(\vv x, \vv k)$ be the top $3\times 3$ corner of the Jacobian matrix $J(\vv x, \vv k)$ --- the same corner occupied by $\tilde{J}_1$ in the modified Jacobian matrix $\tilde{J}$. It is easily checked that the principal minor of Jacobian matrix $J_1(\vv x, \vv k)$ (and therefore the Jacobian matrix $J(\vv x, \vv k)$ itself) is non-singular. Moreover, $-\tilde {J}_1$ is a $P_0$-matrix if and only if $-\tilde {J}$ is a $P_0$-matrix. A simple symbolic  calculation shows that indeed, $-\tilde J_{1}(\vv x, \vv k)$ is a $P_0$-matrix for any positive $\vv x$ and $\vv k$ (in fact, each principal minor of $-\tilde J_{1}$ is a polynomial in $\vv x$ and $\vv k$ with monomials of the same sign). Therefore $\det J(\vv x^*, \vv k)\neq 0$ and $-\tilde J(\vv x^*, \vv k)$ is a $P_0$-matrix for any positive steady state $\vv x^*$ and any positive ${\vv k}$. It follows by Corollary~\ref{cor:main} that the network is delay stable. 
\end{example}

\begin{remark}
Our example is a particular case of the fully open sequestration network  
\begin{equation*}\label{eq:genSeq}
A_i+A_{i+1}\to 0,\quad A_i\rightleftharpoons 0,\quad i\in\{1,\ldots, n\}, \quad A_{n+1}=A_1.
\end{equation*}
One can show that the general sequestration network is delay stable. This  can be proved along the same lines as above (but with substantial extra effort), but also as an immediate application of a theorem in our upcoming paper \cite{our_follow_up_paper_SR_graph}, where we prove that delay stability can be inferred from a certain digraph derived from the network, called the {\em directed species-reaction graph}.
\end{remark} 

\begin{remark}
While a network without positive steady states is trivially  delay stable, it turns out that the general sequestration network does admit positive steady states for any choice of rate constants. Indeed, it is not hard to show that the network is {\em dynamically equivalent}~\cite{craciun2008identifiability} to a reversible network, which is known to have positive steady states for all values of rate constants~\cite{boros}.
\end{remark}

\begin{example}\label{eq:seqKmn}
Next we define a fully open network which includes sequestration reactions for each species, and also includes a transmutation (or synthesis) reaction $A_1 \to  m A_n$ where $m \in \mathbb{N}$. Following the notation in \cite{js15}, we let $K_{m,n} (\tau)$ denote the fully open network 
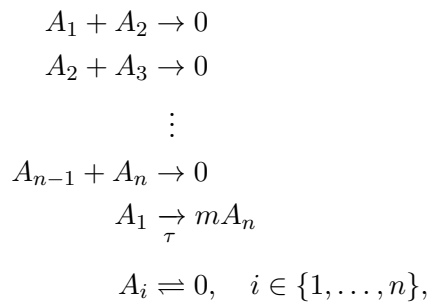
\begin{align*}
A_1 + A_2 &\to 0   \\
A_2 +A_3 &\to 0  \\ 
&\,\,\,\,\vdots   \\
A_{n-1}+A_n &\to 0  \\
A_1 &\xrightarrow[\tau]{} m A_n   \\[5pt]
A_i &\rightleftharpoons 0, \quad i\in\{1,\ldots, n\},
\end{align*} 
where $\tau\ge 0$ is the delay associated to reaction $A_1 \to  m X_n$. Delay-free networks $K_{m,n} (0)$ and more general sequestration-transmutation networks have been analyzed in \cite{Craciun&Feinberg, cf06,sequestration, js15, tw19} for multistationarity and bistability. 

For simplicity, we let  $n=4$ for our calculations below. Let $x_i$ denote the species concentration of $A_i$ for $i=1,2,3,4$. The delay mass action system for the sequestration network  $K_{m,4} (\tau )$ is 
\begin{align*}\label{eq-sn-dde}
    \begin{array}{l}
    \dot{x}_1 = -k_1 x_1x_2 -k_4 x_1 -k_{5} x_1 +k_{9} \\[2pt]
    \dot{x}_2 = -k_{1} x_{1} x_2 -k_2 x_2 x_{3} -k_{6} x_2 +k_{10}   \\[2pt]
    \dot{x}_3 = -k_{2} x_{2} x_3 -k_3 x_3 x_{4} -k_{7} x_3 +k_{11}  \\[2pt]
    \dot{x}_4 = -k_3 x_3 x_4 + m k_4 x_1(t - \tau) -k_{8 } x_4 +k_{12} 
    \end{array}
\end{align*}
and its Jacobian matrix $J(\vv k,\vv x)$ with $\tau=0$ is 
\begin{equation*}\label{eq-seq-jac}
    \begin{bmatrix}
    -k_1 x_2 -k_4-k_{5} & -k_1 x_1 & 0 &0  \\
    -k_1 x_2 & -k_1 x_1 -k_2 x_3 -k_{6} & -k_2x_2 & 0\\
    0 & -k_2 x_3  & -k_2 x_2 -k_3 x_4 -k_{7}  &-k_3 x_3  \\
    m k_4  & 0&  -k_{3}x_4 &  -k_{3} x_{3} -k_{8}
    \end{bmatrix}.
\end{equation*}
One can verify (for example, using {\it  Maple}) that $\det(J({\vv k},{\vv x}))>0$ for all values of the rate constants $\vv k$ and all positive  ${\vv x}$ (in particular for all steady states $\vv x^*$).

The modified Jacobian matrix $\tilde{J} (\vv k,\vv x)$  defined in (\ref{jactilde}) is 
\begin{equation*}\label{eq-seq-mod-jac}
    \begin{bmatrix}
    -k_1 x_2 -k_4-k_{5} & k_1 x_1 & 0 &0  \\
    k_1 x_2 & -k_1 x_1 -k_2 x_3 -k_{6} & k_2x_2 & 0\\
    0 & k_2 x_3  & -k_2 x_2 -k_3 x_4 -k_{7}  &k_3x_3 \\
    m k_4  & 0&  -k_{3}x_4 &  -k_{3} x_{3} -k_{8}
    \end{bmatrix}.
\end{equation*}
Clearly, the  diagonal entries of $\tilde J$ are all negative. 
One can verify using {\it Maple} that all principal minors of size three or less of $-\tilde J$ are positive. On the other hand, $\det (-\tilde{J} ({\vv k},{\vv x}))$ is not immediately positive, except when 
\begin{equation}\label{se-ineq}
    m-1\leq \frac{k_5}{k_4} .
\end{equation}
Therefore by Corollary~\ref{cor:main}, the sequestration network $K_{m,4} (\tau)$ is delay stable if the inequality (\ref{se-ineq}) is satisfied. 

Note that (\ref{se-ineq}) is always satisfied if $m=1$. In fact, $K_{1,n}(\tau)$ is delay stable for any $n$. This is an immediate application of the graph-theoretical result in our upcoming paper~\cite{our_follow_up_paper_SR_graph}.
\end{example}

\begin{example}
\label{ex:notdelaystable}
We conclude with an example that is not delay stable~\cite{mr07}: the detoxification of nitric oxide in bacteria. A toxin $A_1$ leaks into the cell but is degraded by an enzyme $A_2$ (with intermediate complex $A_3$). The enzyme $A_2$ also degrades over time. Finally, in the presence of $A_1$, the gene for producing $A_2$ is promoted by an active promoter $A_5$ and by an inactive promoter $A_4$. The delay mass action system is given by  
    \begin{align*}
        0 &\to A_1 \\
        A_1 + A_2 &\rightleftharpoons A_3 \\
        A_3 &\to A_2 \\
        A_2 &\rightarrow 0 \\
        n A_1 + A_4 &\rightleftharpoons A_5 \\
        A_5  &\xrightarrow[\tau_2, \tau_5]{} A_2 + A_5 ,
    \end{align*}
where only the last reaction is delayed\footnote{As shown, this last reaction seems to violate conservation of mass. However, this simplified model does not show the amino acids needed to produce $A_2$; we assume that there is a high abundance of amino acids available.}, with the enzyme $A_2$ produced after a delay time $\tau_2$ and recovering the promoter $A_5$ after a different delay time $\tau_5$. The two delay times of the system reflect that $A_2$ and $A_5$ require different amount of time before they are able to perform their respective functions. Having two different delay times does not greatly alter the analysis; see Remark~\ref{rmk:difftau}. 

In this example, $-\tilde{J}$ fails to be a $P_0$-matrix (for example the $\{1,4,5\}$ principal minor is negative), and therefore Theorem~\ref{prop:main} does not apply. Moreover, Theorem~\ref{prop:main} also does not apply because $\det(J) = 0$, where $J$ is the Jacobian matrix of the system when $\tau_2 = \tau_5 = 0$. Indeed, the stoichiometric subspace $S$ has the span of $\vv e_4 + \vv e_5$ as its orthogonal complement. Thus, the ODE model has the conservation relation $x_4(t) + x_5(t) = c$, while the delay model has the conservation relation
    \begin{align*}
    x_4(t) + x_5(t) + \int_{t-\tau_5}^t k_8 x_5(s) \, ds = c, 
    \end{align*}
where $k_8$ is the rate constant for the last reaction and $c$ is a constant determined by the initial data (for the ODE or delay model respectively). In fact, for some choices of rate constants and delay parameters, the system is delay unstable and exhibits oscillations~\cite{mr07}. 
\end{example}

\section{Discussion and conclusion}
\label{sec:concl}

Time delays are naturally present in many biochemical and biological processes~\cite{Macdonald1989,Mahaffy1984,Smolen2001}. In this work we have presented a criterion for a chemical reaction network with time delays to be delay stable, that is, any equilibrium is (absolutely) stable for all values of the rate constants and the delay parameters. The criterion for delay stability is algebraic in nature and depends mainly on the signs of the principal minors of what we call the modified Jacobian, a matrix constructed from the Jacobian matrix. The criterion for delay stability is applicable to reaction networks of moderate to large size, since the calculations necessary to show delay stability can be completed using standard symbolic software.

Some future research directions related to delay stability of reaction networks include:
\begin{itemize}
\item  Graph-theoretic conditions for multistationarity in reaction networks have been found in \cite{Banaji_Craciun_2009,cf06,mr07}. These conditions are closely related to the sign of the determinant of the Jacobian matrix for different parameter values. Similarly, the delay stability of reaction networks  with delays depends on the signs of the principal minors of the modified Jacobian matrix. In an upcoming work we will present  similar graph-theoretic condition for delay stability in reaction networks~\cite{our_follow_up_paper_SR_graph}.
\item We have found a sufficient condition for delay stability in reaction networks. Future work will address the question of necessary conditions for delay stability.
\item
The sequestration networks $K_{m,n}$ presented in Example~\ref{eq:seqKmn} have been analyzed extensively recently~\cite{js15,tw19}. In a subsequent work 
their delay stability will be studied for any values of $m$ and $n$.
\end{itemize}

As a final remark, our analysis focused on the stability of steady states in reaction systems with time delays. However, our main result has implications to the study of oscillations (a common dynamical behavior in systems with delay). Namely, if Theorem~\ref{prop:main} applies, then critical delay values (and therefore Hopf bifurcations) are precluded. Thus our result may serve as a necessary condition for the existence of Hopf bifurcations arising from time delays.

\section*{Acknowledgements}
\label{sec:ack}

This work was supported by the National Science Foundation [DMS--1412643, DMS--1517577, DMS--1816238] and Natural Sciences and Engineering Research Council of Canada [PGS-D]. We thank the anonymous reviewers for their comments and suggestions.

\end{document}